\documentclass[11pt]{amsart}
\usepackage{amssymb,amsthm,amsmath,amstext}
\usepackage{mathrsfs}  
\usepackage{bm}        
\usepackage{mathtools} 
\usepackage{color}
\usepackage[bbgreekl]{mathbbol}

\usepackage{cite}

\usepackage[left=1.5in,top=1in,right=1.5in,bottom=1in]{geometry}

\usepackage{graphicx}

\usepackage[all]{xy}
\usepackage{tikz}

\newcommand{\xycenter}[1]{
	\begin{center}
	\mbox{\xymatrix{#1}}
	\end{center}
	}

\theoremstyle{plain}
\newtheorem{theorem}{Theorem}[section]

\newtheorem{proposition}[theorem]{Proposition}
\newtheorem{lemma}[theorem]{Lemma}

\theoremstyle{definition}
\newtheorem{definition}[theorem]{Definition}

\theoremstyle{remark}
\newtheorem{remark}[theorem]{Remark}

\newcommand{\sheaf}[1]{\mathscr{#1}}

\newcommand{\OO}{\sheaf{O}}

\newcommand{\MM}{\sheaf{M}}
\newcommand{\EE}{\sheaf{E}}
\newcommand{\FF}{\sheaf{F}}

\newcommand{\C}{\mathbb C}
\renewcommand{\P}{\mathbb P}

\DeclareMathOperator{\Pic}{\mathrm{Pic}}

\DeclareMathOperator{\coker}{\mathrm{coker}}

\newcommand{\im}{\mathrm{im}}

\newcommand{\id}{\mathrm{id}}

\newcommand{\Pfad}{\mathrm{adj}^{\mathrm{Pf}}}
\newcommand{\Pf}{\mathrm{Pf}}

\newcommand{\wMM}{\widetilde{\MM}}
\newcommand{\bMM}{\overline{\mathcal{M}}}
\newcommand{\oMM}{\mathcal{M}}

\newcommand{\mf}{\mathcal{MF}}

\newcommand{\sk}{\mathcal{S}}
\newcommand{\skss}{\sk^{ss}}
\newcommand{\skps}{\sk^{ps}}

\newcommand{\wskss}{\widetilde{\sk}^{ss}}

\newcommand{\GL}{\mathrm{GL}}
\newcommand{\GLsix}{\GL_6 (\C )}

\newcommand{\cubicsexplicit}{\P\bigl(H^0\bigl(\OO_{\P^4}(3)\bigr)\bigr)}
\newcommand{\cubics}{\mathcal{C}}

\newcommand{\ring}{R}

\usepackage[backref=page]{hyperref}

\begin{document}

\title[Matrix factorizations and intermediate Jacobians]{Matrix factorizations and intermediate Jacobians of cubic threefolds}

\author[B\"ohning]{Christian B\"ohning}\thanks{The first author was supported by the EPSRC New Horizons Grant EP/V047299/1.}
\address{Christian B\"ohning, Mathematics Institute, University of Warwick\\
Coventry CV4 7AL, England}
\email{C.Boehning@warwick.ac.uk}

\author[von Bothmer]{Hans-Christian Graf von Bothmer}
\address{Hans-Christian Graf von Bothmer, Fachbereich Mathematik der Universit\"at Hamburg\\
Bundesstra\ss e 55\\
20146 Hamburg, Germany}
\email{hans.christian.v.bothmer@uni-hamburg.de}

\author[Buhr]{Lukas Buhr}
\address{Lukas Buhr, Institut f\"ur Mathematik\\
Johannes Gutenberg-Universit\"at Mainz\\
Staudingerweg 9\\
55128 Mainz, Germany}
\email{lubuhr@uni-mainz.de}

\date{\today}


\begin{abstract}
Results due to Druel and Beauville show that the blowup of the intermediate Jacobian of a smooth cubic threefold $X$ in the Fano surface of lines can be identified with a moduli space of semistable sheaves of Chern classes $c_1=0, c_2=2, c_3=0$ on $X$. Here we further identify this space with a space of matrix factorizations. This has the advantage that this description naturally generalizes to singular and even reducible cubic threefolds. In this way, given a degeneration of $X$ to a reducible cubic threefold $X_0$, we obtain an associated degeneration of the above moduli spaces of semistable sheaves. 
\end{abstract}

\maketitle

\section{Introduction and recollection about moduli spaces of skew-symmetric $6\times 6$ matrices of linear form}\label{sClassificationSkewSemistable}

Pfaffian representations of cubic threefolds are a classical subject, cf. e.g. Beauville \cite{Beau00}, \cite{Beau02}, Comaschi \cite{Co20}, \cite{Co21}, Manivel-Mezzetti \cite{MaMe05}, Iliev-Markushevich \cite{IM00}; in \cite{BB22-2} we introduced a fibration of algebraic varieties $\wMM \to \cubics$, where $\cubics =\cubicsexplicit$ is the space of cubic hypersurfaces in $\P^4$, which may be viewed as the \emph{universal family of $G$-equivalence classes of Pfaffian representations of cubic threefolds}. Moreover, the fibres are projective. Let us recall how to obtain $\wMM$: we let $\ring$ be the graded polynomial ring over $\C$ in variables $x_0, \dots, x_4$ of weight $1$, and let $\sk = \P \bigl ( (\C^5)^{\vee} \otimes \Lambda^2 \C^6 \bigr)$ be the projective space of skew-symmetric $6\times 6$-matrices with entries linear forms on $\P^4$. This has a natural action by $G=\GLsix$ given by $M\mapsto AMA^t$ for $A\in \GLsix$ and $M\in (\C^5)^{\vee} \otimes \Lambda^2 \C^6$. 
Then
\[
\MM := \skss //G
\]
is the good projective categorical quotient of the locus of semistable points in $\sk$ by this action. Let
\[
\pi \colon \skss \to \MM
\]
be the canonical projection. We define the subset $\skps \subset \skss$ to be the subset of those points whose orbits in $\skss$ are closed. Then look at the incidence correspondence 
\[
\mathcal{T}=\bigl\{  ([M], [F]) \mid \Pf (M) \in (F)\bigr\} \subset \skss \times \cubics 
\]
with its two projections 
\[
\pi_1 \colon \mathcal{T} \to  \skss, \quad \pi_2 \colon \mathcal{T} \to \cubics .
\]

We also denote by $\skss_0 \subset \skss$ the subset consisting of matrices with Pfaffian zero, and let 
\[
\mathcal{T}_0= \pi_1^{-1}(\skss_0) =\skss_0 \times \cubics .
\]
Clearly, $\pi_1$ is one-to-one onto its image outside of the subset $\mathcal{T}_0 \subset \mathcal{T}$ and we denote by $\wskss$ the closure of $\pi_1^{-1} (\skss - \skss_0)$ in $\mathcal{T}$. 

Then the group $G$ acts on $\sk \times \cubics$ if we let it act trivially on $\cubics$ and $\skss \times \cubics$ is the locus of semistable points for this action, $\wskss$ is a $G$-invariant  irreducible closed subset of $\skss \times \cubics$ and the good categorical quotient $(\skss \times \cubics)//G$ is nothing but $\MM \times \cubics$. Then $\wskss$ maps to an irreducible closed subset of $\MM \times \cubics$, which we denote by $\wMM$. It comes with a natural projection $\wMM \to \cubics$.

\medskip

The main results we will show in this article are the following.

\begin{enumerate}
\item
We define a space $\mf$ of certain matrix factorisations, which we call \emph{matrix factorisations of intermediate Jacobian type}. This has an action by a non-reductive group $\Gamma$. Since we were not able to use general methods provided by currently available versions of non-reductive GIT to this example, we define ad hoc the loci of semistable and polystable points, $\mf^{ss}, \mf^{ps}$ in $\mf$, and show that there is a natural morphism $\mf^{ps} \to \wMM$ each nonempty fibre of which is a single $\Gamma$-orbit. Indeed, in a forthcoming paper we will show this morphism is also surjective. In this sense, $\wMM$ can be interpreted as a moduli space of matrix factorisations of intermediate Jacobian type.
\item
For smooth $X$ in $\cubics$, we will show that the fibre $\wMM_X$ of $\wMM\to\cubics$ over $X$ admits a natural bijective morphism onto 
Druel's compactification $\bMM_X$ of the moduli space of stable rank 2 vector bundles with $c_1=0$ and $c_2=2$ on $X$ \cite{Dru00}, \cite{Beau02}, the Maruyama-Druel-Beauville moduli space of equivalence classes of semistable sheaves on $X$ with Chern classes $c_1=0, c_2=2, c_3=0$. Since it is known that $\bMM_X$ is smooth, hence in particular normal, it then follows from Zariski's main theorem that the morphism from $\wMM_X$ is an isomorphism. We mention that $\bMM_X$ is also isomorphic to the intermediate Jacobian of $X$ blown up in the Fano surface of lines \cite{Beau02}. Thus our constructions here allow us to study degenerations of these (birational models of) intermediate Jacobians along with the cubic. Our initial motivation for the present work is that this may help to shed some light on unsolved cycle-theoretic questions about these intermediate Jacobians, such as the representability by algebraic cycles of certain minimal cohomology classes which is intimately linked to the question of stable rationality for smooth cubic threefolds \cite{Voi17}.
\end{enumerate}

\

Lastly, for the reader's convenience, we restate here two results already appearing in \cite{BB22-2}, which we will make repeated use of below. 

\begin{table}
\begin{tabular}{|c|c|c|c| c|}
\hline
& $M$ & $S$ & $Y$ & \\
\hline
 (a)
 &
  $\left(
 \begin{smallmatrix}
       &{l}_{3}&&0&{l}_{0}&{l}_{1}\\
      {-{l}_{3}}&&&{-{l}_{0}}&0&{l}_{2}\\
      &&&{-{l}_{1}}&{-{l}_{2}}&0\\
      0&{l}_{0}&{l}_{1}&&{l}_{4}&\\
      {-{l}_{0}}&0&{l}_{2}&{-{l}_{4}}&&\\
      {-{l}_{1}}&{-{l}_{2}}&0&&&
\end{smallmatrix}
\right)$
&
$\left(
 \begin{smallmatrix}
      {l}_{2}&\\
      {-{l}_{1}}&\\
      {l}_{0}&{l}_{4}\\
      &{l}_{2}\\
      &{-{l}_{1}}\\
      {l}_{3}&{l}_{0}
\end{smallmatrix}
\right)$
&
a smooth conic
&
stable
\\ \hline
 (b)
 &
$\left(
\begin{smallmatrix}
      0&{l}_{0}&{l}_{1}\\
      {-{l}_{0}}&0&{l}_{2}\\
      {-{l}_{1}}&{-{l}_{2}}&0\\
      &&&0&{l}_{2}&{l}_{3}\\
      &&&{-{l}_{2}}&0&{l}_{4}\\
      &&&{-{l}_{3}}&{-{l}_{4}}&0
\end{smallmatrix}
\right)$
&
$\left(
\begin{smallmatrix}
      {l}_{2}&\\
      {-{l}_{1}}&\\
      {l}_{0}&\\
      &{l}_{4}\\
      &{-{l}_{3}}\\
      &{l}_{2}
\end{smallmatrix}
\right)$
 &
 two skew lines
 &
 stable
\\ \hline
 (c)
 &
$\left(
\begin{smallmatrix}
      0&{l}_{0}&{l}_{1}\\
      {-{l}_{0}}&0&{l}_{2}\\
      {-{l}_{1}}&{-{l}_{2}}& 0\\
      &&&0&{l}_{1}&{l}_{2}\\
      &&&{-{l}_{1}}&0&{l}_{3}\\
      &&&{-{l}_{2}}&{-{l}_{3}}&0
\end{smallmatrix}
\right)$
&
$\left(
\begin{smallmatrix}
      {l}_{2}&\\
      {-{l}_{1}}&\\
      {l}_{0}&\\
      &{l}_{3}\\
      &{-{l}_{2}}\\
      &{l}_{1}
\end{smallmatrix}
\right)$
&
 \begin{tabular}{c}
two distinct \\ intersecting lines\\
with an embedded point\\
at the intersection \\
spanning the ambient $\P^4$
\end{tabular}
&
stable
\\ \hline
 (d)
 &
$\left(
\begin{smallmatrix}
&  &  & 0 &l _0 &l_1 \\
 &  &  & -l_0 & 0 & l_2  \\
 &  &  & -l_1 & -l_2 & 0 \\
0 &l _0 &l_1 & & l_3 & l_4\\
-l_0 & 0 & l_2  & -l_3 &  & \\
-l_1 & -l_2 & 0 & -l_4 &  & 
\end{smallmatrix}
\right)$
&
$\left(
\begin{smallmatrix}
      {l}_{2}&\\
      {-{l}_{1}}& -l_4\\
      {l}_{0}&{l}_{3}\\
      &{l}_{2}\\
      &{-{l}_{1}}\\
      &{l}_{0}
\end{smallmatrix}
\right)$
 &
  \begin{tabular}{c}
  a double line lying on \\
a smooth quadric surface\\
\end{tabular}
&
\begin{tabular}{c}
strictly semistable,\\ but not polystable
\end{tabular}
\\ \hline
(e)
 &
$\left(
\begin{smallmatrix}
&  &  & 0 &l _0 &l_1 \\
 &  &  & -l_0 & 0 & l_2  \\
 &  &  & -l_1 & -l_2 & 0 \\
0 &l _0 &l_1 & & l_3 & \\
-l_0 & 0 & l_2  & -l_3 &  & \\
-l_1 & -l_2 & 0 &  &  & 
\end{smallmatrix}
\right)$
&
$\left(
\begin{smallmatrix}
      {l}_{2}&\\
      {-{l}_{1}}& \\
      {l}_{0}&{l}_{3}\\
      &{l}_{2}\\
      &{-{l}_{1}}\\
      &{l}_{0}
\end{smallmatrix}
\right)$
&
  \begin{tabular}{c}
a plane double line \\
with an embedded point,\\
spanning the ambient $\P^4$\\
\end{tabular}
&
\begin{tabular}{c}
strictly semistable,\\ but not polystable
\end{tabular}
\\ \hline
(f) 
&
$\left(
\begin{smallmatrix}
&  &  & 0 &l _0 &l_1 \\
 &  &  & -l_0 & 0 & l_2  \\
 &  &  & -l_1 & -l_2 & 0 \\
0 &l _0 &l_1 & &  & \\
-l_0 & 0 & l_2  &  &  & \\
-l_1 & -l_2 & 0 &  &  & 
\end{smallmatrix}
\right)$
&
$\left(
\begin{smallmatrix}
      {l}_{2}&\\
      {-{l}_{1}}& \\
      {l}_{0}&\\
      &{l}_{2}\\
      &{-{l}_{1}}\\
      &{l}_{0}
\end{smallmatrix}
\right)$
&
\begin{tabular}{c}
 a line \\
together with its \\
full first order\\
infinitesimal \\
neighbourhood
\end{tabular}
&
polystable
\\ \hline
\end{tabular}
\medskip

\caption{Semi-stable matrices $M$ with vanishing Pfaffian}
\label{tPfaffZero}
\end{table}

\begin{theorem}\label{tGeometryM0}
Let $[M]\in \skss$ have vanishing Pfaffian. View $M$ as a map of graded $\ring$-modules
\[
\ring (-1)^{6} \xrightarrow{M} \ring^6.
\]
Let $S$ be a matrix with columns representing a minimal system of generators of the kernel of this map $M$.  Let $Y$ be the rank at most two locus of $M$ with its scheme structure defined by the $4\times 4$ sub-Pfaffians. Then there exists independent linear forms $l_0,\dots,l_4$ and matrices $B \in \mathrm{GL_6}(\C)$ and $B'  \in \mathrm{GL_2}(\C)$ such that after making the replacements 
\begin{align*}
	M & \mapsto B^t M B \\
	S & \mapsto B^{-1} S B'
\end{align*}
we have one of the cases in Table \ref{tPfaffZero}. Moreover, the stability type of $M$ is as described in the last column of Table \ref{tPfaffZero}.
\end{theorem}

\begin{proposition}\label{pSameInformation}
Let $M$ and $S$ be matrices as in Table \ref{tPfaffZero}. Then $M$ represents the syzygy module of $S^t$. If $M'$ is a another skew symmetric $6 \times 6$ matrix with linear forms representing the syzygy module of $S^t$ then 
$[M']$ is in the $G$-orbit of $[M]$.  Furthermore the ideal generated by the $2\times 2$ minors of $S$ is equal to the one generated by the $4 \times 4$ Pfaffians of $M$. 
\end{proposition}

\section{Matrix factorizations of intermediate Jacobian type}\label{sMatrixFactorizations}

In this section we identify $\wMM$ with a space of matrix factorizations. 

\medskip

Consider the following three projective spaces: 
\begin{enumerate}
\item
the projective space $\P_f= \cubics$ parametrising cubics in $\P^4$; for a nonzero homogeneous cubic polynomial $F$ in $x_0, \dots , x_4$, we will denote by $[F]$ the corresponding element of $\P_f$. 
\item
the projective space $\P_a$ parametrising skew-symmetric maps 
\[
2 \OO_{\P^4}(1)  \oplus 6 \OO_{\P^4} (2) 
	\xrightarrow{A} 
	2 \OO_{\P^4} (4)\oplus 6\OO_{\P^4}(3) 
\]
up to homotheties. We will write $[A]$ for the class of such a map $A$ in $\P_a$. We use the notation 
\[
	A = \begin{pmatrix}
		A_3 & A_2 \\
		-A_2^t & A_1
		\end{pmatrix}
\]
where the subscript denotes the homogeneous degree of the entries of each matrix;
\item
the projective space $\P_b$ parametrising skew-symmetric maps 
\[
 2 \OO_{\P^4} (1) \oplus 6 \OO_{\P^4}
	\xrightarrow{B} 
	2 \OO_{\P^4}(1)  \oplus 6 \OO_{\P^4} (2) 
\]
up to homotheties. Again we write $[B]$ for the class of such a map $B$ in $\P_b$ and 
\[
B = \begin{pmatrix}
		B_0 & B_1^t\\
		-B_1 & B_2
		\end{pmatrix}. 
\]
\end{enumerate}

\begin{definition}\label{dMFIntermediateJacobian}
Let $\mf$ be the space of triples $([A],[B],[F])\in \P_a\times \P_b\times\P_f$ such that the following hold:
\begin{enumerate}
\item
$AB\neq 0$, $BA \neq 0$, $\Pf(A)\neq 0$, $\Pf (B)\neq 0$. 
\item
$[AB] =[F\cdot \mathrm{id}]$ and $[BA]= [F\cdot \mathrm{id}]$, where the equalities (and identity maps) are to be understood in the appropriate projective spaces. \
\item
$[\Pf (A)]=[\Pf(B)] =[F^2]$. 
\end{enumerate} 
Note that condition $a)$ defines an open subvariety of $\P_a\times \P_b\times\P_f$, and $b)$ and $c)$ define a closed subvariety inside this. Hence $\mf$ is a locally closed subvariety of $\P_a\times \P_b\times\P_f$. We call elements in $\mf$ {\sl matrix factorizations of intermediate Jacobian type}.
\end{definition}

\begin{definition}\label{dGroupActionMF}
Consider the (non-reductive) group of automorphism $\Gamma$ of the graded free bundle $2 \OO_{\P^4}(1) \oplus 6 \OO_{\P^4}$. $\Gamma$ acts on $\mf$ by the rule
\[
\gamma \cdot ([A],[B],[F]) = \bigl( [\gamma A \gamma^t], [(\gamma^t)^{-1}B \gamma^{-1}], [F] \bigr). 
\]

The situation is summarised in the following commutative diagram
\xycenter{
	2 \OO_{\P^4} (1) \oplus 6 \OO_{\P^4} 
		\ar[r]^-B \ar[d]^{\gamma}
	&
	2 \OO_{\P^4}(1) \oplus 6 \OO_{\P^4} (2) 
	\ar[r]^-A
	\ar[d]^{(\gamma^{-1})^t}
	&
	2 \OO_{\P^4} (4) \oplus 6\OO_{\P^4}(3) 
	\ar[d]^{\gamma}
	\\
	2 \OO_{\P^4} (1) \oplus 6 \OO_{\P^4} 
	\ar[r]^-{B'} 
	&
	2 \OO_{\P^4}(1) \oplus 6 \OO_{\P^4} (2) 
	\ar[r]^-{A'}
	&
	2 \OO_{\P^4} (4) \oplus 6\OO_{\P^4}(3) 
	}
\end{definition}

\begin{definition}\label{dMFStabilityProp}
We denote by $\mf^{ss}$ and $\mf^{ps}$ the loci inside $\mf$ where $[A_1]$ is semistable and polystable, respectively.
\end{definition}

\begin{lemma}\label{lMFNormal}
Let $([A], [B], [F])\in \mf$. Then 
\[
\Pf (A_1)\in (F). 
\]
More precisely:
\begin{enumerate}
\item[(A)]
 If $B_0 \not= 0$ then, $([A],[B],[F])$ is in the same $\Gamma$-orbit as a matrix factorization $([A'],[B]',[F])$ with
\[
	A' = \begin{pmatrix}
			0 & F & \\
			-F & 0 & \\
			&& A'_1
		\end{pmatrix}
\text{ and }
	B' = \begin{pmatrix}
			0 & 1 & \\
			-1 & 0 & \\
			&& \Pfad {A'_1}
		\end{pmatrix}
\]
where $[A_1']$ is in the same $G$-orbit as $[A_1]$. 
In this case, $(\Pf (A_1))= (F)$.  
\item[(B)]
If $B_0=0$, then $\Pf (A_1)=0$. Moreover, if $A_1$ is semistable, then $([A],[B],[F])$ is in the same $\Gamma$-orbit as a matrix factorization $([A'],[B'],[F'])$ with 
$A'_1 = M$ and $B'_1 = S$ where $M$ and $S$ are as in one of the cases in Table \ref{tPfaffZero}.

\end{enumerate}
In particular, $\Pf (A_1) =0$ if and only if $B_0=0$.
\end{lemma}

\begin{proof}

\fbox{Case (A):} If $B_0$ is nonzero, it is invertible since it is skew. Therefore there exists a $\Gamma$-translate $([A'], [B'], [F])\in \mf$ such that $A_1=A_1'$ and 
\[
	B_0' = \begin{pmatrix} 0 & 1 \\ -1 & 0 \end{pmatrix}
\]
and $B_1' = 0$. But then
\[
	[F \cdot \id ]
	= [A' B']  
	= \left[ \begin{pmatrix}
		A_3' & A_2' \\
		-(A_2')^t & A_1'
		\end{pmatrix}
	\begin{pmatrix}
		B_0' & 0\\
		0 & B_2'
		\end{pmatrix} \right].
\]
In particular $(A_2')^t B_0' = 0$. Since $B_0'$ is invertible, it follows that $A_2' = 0$. So
\[
	[F \cdot \id] = [A' B' ]
	= \left[ \begin{pmatrix}
		A_3'B_0'&  0\\
		0 & A_1'B_2'
		\end{pmatrix}\right].
\]
This implies
\[
	[A_3' ]= \left[\begin{pmatrix}0 & F \\ -F & 0 \end{pmatrix}\right]
\]
and since $[\Pf (A')]= [F^2]$ by assumption, we must consequently have $[\Pf (A_1)]= [\Pf (A_1')]= [F]$ in this case. 

We have $[A_1'B_2']= [F\cdot \mathrm{id}_{6\times 6}]$. Moreover, $[\Pf (A_1')] = [F]$, $[\Pf (B_2')] = [\Pf (B)] = [F^2]$ and then $B_2'$ must be proportional to the Pfaffian adjoint of $A_1'$. 

In summary, what we have achieved at this point is reduction of $A$ and $B$ to
\[
	A'= \begin{pmatrix}
			0 & \lambda F & \\
			-\lambda F & 0 & \\
			&& A'_1
		\end{pmatrix}
\text{ and }
	B'= \begin{pmatrix}
			0 & 1 & \\
			-1 & 0 & \\
			&& \mu\Pfad {A'_1}
		\end{pmatrix}
\]
with $\lambda, \mu \in \C^*$. Since $A'B'$ is a multiple of $F$, we have $\lambda=\mu$. Now acting with the element $\gamma \in \Gamma$ given by
\[
\gamma =\mathrm{diag} \left( (\sqrt{\lambda})^{-1}, (\sqrt{\lambda})^{-1}, 1, 1, 1, 1, 1, 1 \right)
\]
and noting that $[B'] = [\lambda B']$, we obtain $A'$ and $B'$ as claimed in part (A) of the Lemma (treating $A'$ and $B'$ as dynamical variables in this last step).

\medskip

\fbox{Case (B):} If $B_0=0$, we have $A_1B_1 = 0$ since $[AB] =[F\cdot \mathrm{id}]$. Since $B_1 \not=0$ (otherwise we would have $\Pf(B) = 0$ contradicting $[\Pf (B)] =[F^2]$) the skew symmetric $6 \times 6$ matrix $A_1$ can not have generic rank $6$. Therefore $\Pf(A_1) = 0$.

If we assume in addition that $A_1$ is semi-stable, we can appeal to the classification in Table \ref{tPfaffZero} and conclude that, after acting by an element in $\Gamma$, we can assume $A_1 = M$ with $M$ as in that table. Let $S$ be the matrix associated to $M$ in Table \ref{tPfaffZero}. Since the syzygy module of $M=A_1$ is generated by the columns of $S$ and since $A_1B_1 = 0$ the columns of $B_1$ must be generated by the columns of $S$. The columns of $B_1$ cannot be dependent since $\Pf(B) \not=0$. Since the entries of $B_1$ and $S$ are linear and both matrices have two columns this implies that we may assume $B_1 = S$ after acting with another element in $\Gamma$. 
\end{proof}

\begin{lemma}\label{lWellDefined}
Suppose that $([A],[B],[F]) \in \mf^{ss}$ and $\Pf(A_1) = 0$. 
\begin{enumerate}
\item
We have that $X=V(F)$ contains the rank $\le 2$ locus of $A_1$ set-theoretically.
\item
If $A_1$ is of types (d) or (e) in Table \ref{tPfaffZero} and let $Y$ be the subscheme defined by the $4\times 4$ sub-Pfaffians of $A_1$, we can assert something stronger: in that case $X=V(F)$ scheme-theoretically contains the unique degree $2$ curve contained in $Y$.
\end{enumerate}
\end{lemma}

\begin{proof}
For $a)$, suppose $A_1$ has rank at most $2$ at a point $P\in \P^4$. Observe quite generally that if $A_1$ has rank $r$, then $A$ has rank at most $r+4$ because $A$ can be obtained from $A_1$ by first expanding $A_1$ by two length $6$ columns to a $6\times 8$ matrix, and afterwards, adding two length $8$ rows to get $A$; and for each added row or column the rank can increase by at most $1$. 
Hence if $A_1$ has rank at most $2$, $A$ can never have full rank, hence its determinant, which is $F^2$, is zero at $P$, hence $X=V(F)$ contains $P$. 

Now suppose we are under the assumptions in $b)$. Then, by passing to another element in the same $\Gamma$-orbit, by Lemma \ref{lMFNormal} we can assume that $A_1$ is one of the matrices $M$ in types (d) or (e) in Table \ref{tPfaffZero}, $B_0=0$ and $B_1=S$. Since $[AB]=[F]$, we obtain
\[
[F\id_{2\times 2}]= [ A_2B_1] =[A_2 S].
\]
Denoting
\[
A_2 = (Q_{ij})_{1\le i \le 2, 1\le j \le 6}
\]
we get in case (d)
\begin{align*}
(Q_{11}, \dots, Q_{16}) \cdot (l_2, -l_1, l_0, 0, 0, 0)^t &= F\\
(Q_{11}, \dots, Q_{16}) \cdot (0, -l_4, l_3, l_2, -l_1, l_0)^t &= 0
\end{align*}
and in case (e) 
\begin{align}
(Q_{11}, \dots, Q_{16}) \cdot (l_2, -l_1, l_0, 0, 0, 0)^t &= F \label{fCaseE1} \\ 
(Q_{11}, \dots, Q_{16} )\cdot (0, 0, l_3, l_2, -l_1, l_0)^t &= 0. \label{fCaseE2}
\end{align}

The argument in both cases can now be finished similarly, for simplicity we only give details for case (e). Relation (\ref{fCaseE2}) above shows that $(Q_{11}, \dots, Q_{16} )$ is a degree $2$ syzygy of $(0, 0, l_3, l_2, -l_1, l_0)$, and the module of all syzygies is generated by the columns of
\[
\begin{pmatrix}
     1&0&0&0&0&0&0&0\\
     0&1&0&0&0&0&0&0\\
     0&0&{l}_{2}&{l}_{1}&0&{l}_{0}&0&0\\
     0&0&{-{l}_{3}}&0&{l}_{1}&0&{l}_{0}&0\\
     0&0&0&{l}_{3}&{l}_{2}&0&0&{l}_{0}\\
     0&0&0&0&0&{-{l}_{3}}&{-{l}_{2}}&{l}_{1}\end{pmatrix}.
\]
Therefore there exist linear forms $m_0, m_1, m_2$ auch that 
\[
Q_{13}= l_2 m_2 +l_1m_1 +l_0m_0.
\]
Substituting this into formula (\ref{fCaseE1}) gives
\[
F= Q_{11}l_2 - Q_{12}l_1 + l_0 (l_2 m_2 +l_1m_1 +l_0m_0)
\]
and hence $F \in (l_1, l_2, l_0^2)$ which is the ideal of the unique double line contained in the subscheme defined by the $4\times 4$ sub-Pfaffians of $A_1$. 
The computation in case (d) is similar. This finishes the proof.
\end{proof}

\begin{proposition}\label{pWellDefinedMap}
The map
\[
\psi \colon \mf^{ss} \to \wskss
\]
sending $([A],[B],[F])$ to $([A_1], [F])$ is well-defined. Hence one obtains a map
\[
\overline{\psi}\colon \mf^{ss} \to \wMM
\]
by post-composing with the map to the quotient. Moreover, $\overline{\psi}$ is constant on $\Gamma$-orbits.
\end{proposition}

\begin{proof}
By \cite[Theorem 2.1]{BB22-2} it suffices to show the following two assertions to prove that $\psi$ is well-defined:
\begin{enumerate}
\item
If $([A],[B],[F]) \in \mf^{ss}$ and $\Pf(A_1) \neq 0$, then the ideal generated by $F$ coincides with the ideal generated by $\Pf(A_1)$.
\item
If $([A],[B],[F]) \in \mf^{ss}$ and $\Pf(A_1) = 0$, then the subscheme $\overline{Y}$ defined by the $4\times4$ sub-Pfaffians of $A_1$ and $F$ contains a degree $2$ curve. In other words, in cases (a) -(e) of Table \ref{tPfaffZero} we need to show that $X= V(F)$ contains the unique degree $2$ curve contained in the scheme $Y$, and in case (f) we need to show that $X$ contains the rank $2$ locus of $A_1$. 
\end{enumerate}

Notice that under the assumptions of $a)$, Lemma \ref{lMFNormal} says we are in case (A) of that Lemma, and the conclusion that the ideal generated by $F$ coincides with the ideal generated by $\Pf(A_1)$ is true. Hence it suffices to prove $b)$, but this is nothing but the assertion of Lemma \ref{lWellDefined}. The remaining statements of the Proposition are clear.
\end{proof}

\begin{theorem}\label{tFibresMF}
Each fibre of the map
\[
\overline{\psi}\mid_{\mf^{ps}} \colon \mf^{ps} \to \wMM
\]
is a single $\Gamma$-orbit. 
\end{theorem}

\begin{proof}
Let $x\in \wMM$ be an element represented by a pair $([A_1], [F])$ with $A_1$ polystable. We show that if $([A], [B], [F])$ and $([A'], [B'], [F])$ are two elements in $\mf^{ps}$ mapping down to $x$, then they are in the same $\Gamma$-orbit. 
We distinguish two cases: $\Pf (A_1)\neq 0$ and $\Pf (A_1)=0$. 

\medskip

\fbox{Case 1: $\Pf (A_1)\neq 0$}: In this case, both $([A], [B], [F])$ and $([A'], [B'], [F])$ can be brought to the normal form in Lemma \ref{lMFNormal}, (A). 

\medskip

\fbox{Case 2: $\Pf (A_1) = 0$}: Here we can assume $A_1=A_1'=M$ and $B_1=B_1'=S$ as in (a), (b), (c) or (f) in Table \ref{tPfaffZero} by Lemma \ref{lMFNormal}, (B). Furthermore we can assume 
\[
	AB = A'B' = F \cdot \id.
\]
Indeed, the definition of matrix factorization now gives
\[
	[AB] = [A'B'] = [F \cdot \id]
\]
and we can assume
\[
	AB = F \cdot \id \quad \text{and}  \quad A'B' = \lambda F \cdot \id
\]
Replacing $B'$ by $\lambda^{-1} B'$ and operating with 
\[
	L = \begin{pmatrix} \lambda^{-1} \id_2 & 0 \\ 0 &\id_4 \end{pmatrix}
\]
on the primed matrix factorization, 
we keep $A_1 =  A_1'= M$, $B_1 = B_1'=  S$ and get in addition
\[
	AB = A'B' = F \cdot \id. \quad \quad (\ast)
\]

The operation of the unipotent radical of $\Gamma$ is now
\begin{align*}
\begin{pmatrix}
	\id & u \\
	0 & \id
\end{pmatrix}\begin{pmatrix}
	A_3 & A_2 \\
	-A_2^t & M
\end{pmatrix}\begin{pmatrix}
	\id & 0 \\
	u^t & \id
\end{pmatrix} &= \begin{pmatrix}
	A_3-uA_2^t+A_2u^t+uMu^t & A_2+uM \\
	-A_2^t + Mu^t & M
\end{pmatrix}\\
\begin{pmatrix}
	\id & 0 \\
	-u^t & \id
\end{pmatrix}\begin{pmatrix}
	0 & S^t \\
	-S & B_2
\end{pmatrix}\begin{pmatrix}
	\id & -u \\
	0 & \id
\end{pmatrix} &= \begin{pmatrix}
	0 & S^t \\
	-S & B_2+Su-u^tS^t
\end{pmatrix}
\end{align*}
Writing out the condition $(*)$ we get
\[
	\begin{pmatrix}
	-A_2S & A_3S^t+A_2B_2\\
	0 &  -A_2^tS^t + MB_2
	\end{pmatrix}
	=
	\begin{pmatrix}
	-A'_2S & A_3'S^t+A_2'B_2' \\
	0 &  -A_2^tS^t + MB'_2
	\end{pmatrix}
	= 
	\begin{pmatrix}
	\id \cdot F & 0 \\
	0 & \id \cdot F
	\end{pmatrix}
\]
In particular $( A_2- A_2')S=0$, i.e. that $ A_2-A_2'$ is a quadratic syzygy of $S$. Since all such syzygies are generated by the rows of $M$ by Proposition \ref{pSameInformation}, we can find a $2\times 6$ matrix $u$ of linear forms such that 
$$A_2'-A_2=uM \iff A_2' = A_2+uM$$
Operating with 
$$
U = \begin{pmatrix}
	\id & u \\
	0 & \id
\end{pmatrix}
$$
on $A$ and $B$ we can assume in addition $A_2 = A_2'$.

Now we argue similarly for $B_2$. By $(\ast)$ we have
$$
- A_2^tS^t+MB_2 = F\cdot \id = -A_2^tS^t + MB_2'
$$
and hence $M(B_2'- B_2)=0$. Thus we can choose a $2 \times 6$ matrix of linear forms $v$ such that $Sv=B_2'- B_2$. By skew-symmetry we find
$$
	B_2'=B_2+\frac{1}{2}(-v^tS^t+Sv).
$$ 
Acting with
$$
V = \begin{pmatrix}
\id & \frac{1}{2}v\\
0 & \id
\end{pmatrix}
$$
on $A$ and $B$ we obtain $B_2=B_2'$ but possibly loose the equation $A_2 = A_2'$. Still $(*)$ now gives
$$
-A_2^tS^t+MB_2 = F\cdot \id = -(A_2')^tS^t + MB_2
$$
and hence $S(A_2-A_2')=0$. But $S$ does not have any syzygies on this side, so $A_2 = A_2'$.

Finally $(*)$ now gives
$$
SA_3 + B_2A_2^t =0=SA_3'+B_2A_2^t
$$
i.e. $S(A_3-A_3')=0$ and $S$ has no syzygies on this side. Hence also $A_3 = A_3'$.
\end{proof}

\begin{remark}\label{rSurjectivityMF}
It can be shown that the map $\overline{\psi}\mid_{\mf^{ps}} \colon \mf^{ps} \to \wMM$ is surjective, but the proof uses a variant of the Shamash construction, which we will discuss in its proper context in a different article.
\end{remark}

\section{The connection to the intermediate Jacobian}\label{sIntermediateJacobians}

In this section we fix a smooth cubic threefold $X \subset \P^4$ and a cubic polynomial $F \in \C[x_0,\dots,x_4]$ defining $X$. We denote by $J_X$ the intermediate Jacobian, and consider 
\[
	\wMM_X := \{ ([M], [F]) \, \mid \, V(F) = X\} \subset \wMM 
\]
with its reduced structure.

Given $([M],[F]) \in \wMM_X$, we consider the image $\FF_M$ of the map given by the Pfaffian adjoint matrix:
\[
 6 \OO_X(-1) \xrightarrow{\Pfad(M)} 6 \OO_X(1) . 
\]
Here $\Pfad(M)$ is the skew matrix whose $(i,j)$-entry for $i<j$ is 
\[
(-1)^{i+j} \Pf (M_{ij})
\]
where $M_{ij}$ is the matrix obtained from $M$ by deleting rows and columns with indices $i, j$.

We now want to prove that there is an isomorphism from $\wMM_X$ to 
Druel's compactification $\bMM_X$ of the moduli space of stable rank 2 vector bundles with $c_1=0$ and $c_2=2$ on $X$ sending $([M],[F])$ to $\FF_M$. 
For this we recall some properties of $\bMM_X$ from \cite{Dru00} following \cite{Beau02}.

\medskip

$\bMM_X$ is smooth and connected, and contains a nonempty open subset $\oMM_X$ corresponding to stable rank 2 vector bundles $\EE$ with $c_1=0$ and $c_2=2$ on $X$. Every such $\EE$ sits in an exact sequence 
\[
0 \to 6 \OO_{\P^4}(-2) \xrightarrow{M} 6 \OO_{\P^4}(-1) \to \EE \to 0 
\]
with $M$ a skew $6\times 6$ matrix with $X = V\bigl( \Pf (M) \bigr)$, and conversely, every cokernel $\EE_M$ of a skew map $M$ with $X = V\bigl( \Pf (M) \bigr)$ as in the above short exact sequence is a vector bundle of this type. From this one sees that $\oMM_X$ is isomorphic to the GIT quotient of the space of such skew matrices up to congruence. Points in $\bMM_X - \oMM_X$ correspond to isomorphism classes of polystable sheaves with $c_1=0, c_2=2, c_3=0$ on $X$. Moreover, 
\[
\bMM_X - \oMM_X = \mathcal{A}\cup \mathcal{B}
\]
with $\mathcal{A}$ an irreducible locally closed codimension $1$ subvariety, and $\mathcal{B}$ an irreducible closed codimension $1$ subvariety containing all points of $\overline{\mathcal{A}}-\mathcal{A}$. Representatives of points in $\mathcal{A}$ and $\mathcal{B}$ can be described as follows.

\begin{itemize}
\item[($\alpha$)] Let $C$ be a smooth conic in $X$, $L$ the positive generator of $\Pic(C)$ (so that $\OO_X(1)|_C
\simeq L^{\otimes 2}$ ). Let $\EE$ be the kernel of the canonical evaluation map $H^0(C, L) \otimes \OO_X \to L $. Then $\EE$ is a torsion free sheaf, with $c_1(\EE) = c_3(\EE) = 0$ and $c_2(\EE) = [C]$. 
\item[($\beta$)] Let $l,l'$ be two lines in $X$ (possibly equal), and let $I_l,I_{l'}$ be their ideal sheaves. Then the sheaf $I_l \oplus I_{l'}$ is a torsion free sheaf with $c_1(E) = c_3(E) = 0$ and $c_2(E) = [l] + [l']$ . 
\end{itemize}

\medskip

We will now show that there is a morphism 
\[
	\wMM_X \xrightarrow{\im\, \Pfad}  \bMM_X
\]
sending $([M],[F])$ to $\FF_M$. We then consider the morphism $\widetilde{c}_2$ defined by the diagram
\xycenter{
	\wMM_X \ar[r]^{\im\, \Pfad} \ar[dr]_{\widetilde{c}_2} & \bMM_X  \ar[d]^{c_2}\\
	& J_X^{(2)}
}
Here we think of the intermediate Jacobian $J_X$ as being parametrised, using the Abel-Jacobi map, by the group of $1$-cycles on $X$ that are homologically equivalent to zero, modulo those that are rationally equivalent to zero, and of $J_X^{(2)}$ as the translate of $J_X$ by the class of twice a line. See \cite{Gri84} for details concerning these definitions.  Moreover, the Chern classes are understood to take values in the Chow groups of cycles modulo rational equivalence, as in Grothendieck's sense. 

\begin{lemma}\label{lDruelNonZero}
If $\Pf(M) \not=0$ then $\EE_M \simeq \FF_M$.
\end{lemma}

\begin{proof}
We have
 \[ 
 \Pfad(M)\cdot M =M \cdot \Pfad(M) =\Pf (M) \cdot\mathrm{Id},
 \] 
 hence on $X$ we get a $2$-periodic exact sequence
 \[
\dots \xrightarrow{} 6 \OO_X(-2)\xrightarrow{M} 6 \OO_X(-1) \xrightarrow{\Pfad(M)} 6 \OO_X(1) \xrightarrow{M} 6\OO_X (2) \xrightarrow{} \dots 
 \]
 hence
 \begin{gather*}
 \EE_M=\mathrm{coker}\left(6 \OO_X(-2)\xrightarrow{M} 6 \OO_X(-1)  \right) \simeq \ker \left( 6 \OO_X(1) \xrightarrow{M} 6\OO_X (2) \right)\\
  \simeq \im \left( 6 \OO_X(-1) \xrightarrow{\Pfad(M)} 6 \OO_X(1) \right) = \FF_M .
 \end{gather*}
 \end{proof}

\noindent
For the case $\Pf(M) = 0$, we need some preparations. 

\begin{lemma}\label{lDruelPrep}
Let $M$ and $S$ be as in Table \ref{tPfaffZero}. Then
\[
	\Pfad(M) = S I S^t
\]
with $I = \left(\begin{smallmatrix} 0 & 1 \\ -1 & 0\end{smallmatrix}\right)$ in cases (a), (d), (e), (f) and $I = \left(\begin{smallmatrix} 0 &  -1 \\ 1 & 0 \end{smallmatrix}\right)$ in cases (b) and (c). 

\noindent
Moreover, letting  
\[
\FF_{S} =  \im \left( 6 \OO_X(-1) \xrightarrow{S^t} 2 \OO_X \right) .
\]
we have
\[
	\FF_M \simeq  \FF_S.
\]
\end{lemma}

\begin{proof}
The first assertion is a direct computation \cite[SISt.m2]{BB-M2}. Moreover, it can be checked \cite[SISt.m2]{BB-M2} that $S$ is injective on $\P^4$ and thus in particular on  $X$. 

We have a commutative diagram

\xycenter{
6 \OO_X(-1) \ar[rr]^{\Pfad(M)} \ar[d]^{S^t}  &  & 6 \OO_X(1) \\
2 \OO_X \ar[rr]^{I} &  & 2 \OO_X \ar[u]^S .
}
Since $S\circ I$ is injective, we see that $\FF_M \simeq  \FF_S$. 
\end{proof}

We now go through the polystable cases (a), (b), (c) and (f) in Table \ref{tPfaffZero}. 

\begin{lemma}\label{lConic}
Let  $([M],[F])\in \wMM_X$ where $M$ is of type $(a)$ in Table \ref{tPfaffZero}. Let $S$ be the corresponding syzygy matrix. Then $\FF_S$ is isomorphic to a sheaf $\EE$ given by the construction in case $(\alpha)$ above. 
Moreover, any such sheaf $\EE$ arises from a unique pair $([M],[F])\in \wMM_X$ in this way. 
\end{lemma}

\begin{proof}

In case $(a)$ we have
\[
S^t = \begin{pmatrix}
	l_2 & -l_1 & l_0 &&& l_3 \\
	&& l_4 & l_2 & -l_1 & l_0
\end{pmatrix}
\]
with the $l_i$ independent linear forms. $S^t$ has rank $1$ on a smooth conic $C$ in $X$, defined by
$l_2 = l_1 = l_3l_4-l_0^2 = 0$, and nowhere rank $0$. Therefore 
\[
\coker \left( 6 \OO_{\P^4}(-1) \xrightarrow{S^t} 2 \OO_{\P^4} \right)
\]
is a line bundle $L$ supported on 
$C$. Restricting to the $\P^2$ defined by $l_2=l_1=0$ we obtain an exact sequence
\[
	0 \to 2 \OO_{\P^2}(-1)
	\xrightarrow{
	\left(
	\begin{smallmatrix}
	l_3 & l_0 \\
	l_0 & l_4
	\end{smallmatrix}
	\right)
	}
	2 \OO_{\P^2}
	\to 
	L
	\to
	0.
\]	
By for example \cite[Proposition 3.1 (b)]{Beau00} applied to $L(-1)$ the degree of $L$ must be $1$. 
It is therefore the positive generator of $\Pic(C)$. We now consider the sequence on $X$: 
\[
 6 \OO_{X}(-1) \xrightarrow{S^t} 2 \OO_{X} \xrightarrow{} L \rightarrow 0.
\]
The map $2 \OO_{X} \to L$ gives linearly independent global sections of $L$; since $\deg L=1$, $H^0(C, L)$ has dimension $2$ and thus the map 
can be identified with the evaluation map $H^0(C, L)\otimes\OO_X \to L$. Therefore, $\FF_S$ is isomorphic to the sheaf $\EE$  given by the construction in case $(\alpha)$ for the conic $C$. 

\medskip

Any sheaf $\EE$ of this type arises like this because all conics in $\P^4$ form one $\mathrm{GL}_5 (\C)$-orbit, therefore all such sheaves $L$ as in $(\alpha)$ arise as cokernels of matrices $S^t$ as in (a). Then let $M$ be the corresponding matrix as in in Table \ref{tPfaffZero}, (a). 

For the uniqueness statement notice that $S^t$ represents the map $\varphi_1$ in the minimal free resolution
\[
\dots \rightarrow F_2 \xrightarrow{\varphi_2} F_1 \xrightarrow{\varphi_1} F_0 \rightarrow \Gamma_* (L) \rightarrow 0
\]
 of the associated graded module 
\[
\Gamma_* (L) =\bigoplus_{n\ge 0} H^0 \bigl( L(n) \bigr) =\bigoplus_{n \ge 0} H^0 \bigl( \OO_{\P^1}(1+2n)\bigr).
\]
Then one uses Proposition \ref{pSameInformation} to conclude that this determines $M$ up the action of $G$.
\end{proof}

\begin{lemma}\label{lTwoLines}
Let $S$ be as in Table \ref{tPfaffZero}, cases $(b), (c)$ or $(f)$, and let $l,l'$ be the two lines (possibly equal) where  $S$ drops rank. Then
$\FF_S = I_{l} \oplus I_{l'}$ is the sheaf associated to $l,l'$ as in case $(\beta)$ above. Conversely, any sheaf as in case $(\beta )$ arises from a unique pair $([M],[F])\in \wMM_X$ in this way. 
\end{lemma}

\begin{proof}
In cases $(b)$, $(c)$ and $(f)$ we have
\[
S^t = \begin{pmatrix}
	l_2 & -l_1 & l_0 &&&  \\
	&& & m_2 & -m_1 & m_0
\end{pmatrix}
\]
with  $l = V(l_1,l_2,l_3)$ and $l' = V(m_1,m_2,m_3)$ lines in $X \subset \P^4$. The first claim follows.

For the converse first notice that there are three $\mathrm{GL}_6 (\C )$-orbits of pairs of lines in $\P^4$ corresponding to the cases $(b), (c), (f)$. Therefore every sheaf of type $(\beta)$ arises in this way. For the uniqueness notice that $\FF_S$ in this case determines $S$ up to row and column operations, hence $M$ up to an action of $G$ by Proposition \ref{pSameInformation}.
\end{proof}

\begin{lemma}\label{lNotPolystable}
Let $S$ be as in Table \ref{tPfaffZero}, cases $(d)$ or $(e)$, and let $l$ be the line where  $S$ drops rank. Then
$\FF_S$ is a nontrivial extension of the ideal sheaf $I_l$ by itself. 
\end{lemma}

\begin{proof}
In cases $(d)$ and $(e)$ we have
\[
S^t = \begin{pmatrix}
	l_2 & -l_1 & l_0 &&&  \\
  m_2 & -m_1 & m_0 & l_2 & -l_1 & l_0 
\end{pmatrix}
\]
with  $l = V(l_0,l_1,l_2)$ and $m_0, m_1, m_2$ certain linear forms not all equal to zero. 

Let $A_X =\C [x_0,\dots , x_4]/(F)$ be the homogeneous coordinate ring and consider $S^t$ as a map
\[
 6 A_X (-1) \xrightarrow{S^t}  A_X \oplus A_X .
\]
Here the first copy of $A_X$ in $A_X \oplus A_X$ corresponds to the first row of $S^t$ above, the second copy of $A_X$ to the second row. Projecting onto the first copy of $A_X$ we see that $\mathrm{im}\, S^t$ sits in an extension
\[
0 \to K \to \mathrm{im}\, S^t \to I_l \to 0
\]
where $I_l= (l_0,l_1,l_2)$ and $K$ contains $I_l$. We need to show that in cases (d) and (e) we indeed have $K=I_l$. Denoting $e_1, e_2$ a basis of $A_X \oplus A_X$, this amounts to proving that whenever an $A_X$-linear combination 
\[
\alpha (l_2e_1 +m_2e_2 ) + \beta ( l_1e_1 +m_1e_2 ) + \gamma ( l_0 e_1 + m_0 e_2) 
\]
is such that 
\[
\alpha m_2e_2  + \beta m_1e_2  + \gamma m_0 e_2 = 0
\]
then
\[
\alpha m_2  + \beta m_1  + \gamma m_0 
\]
is already in the ideal generated by $(l_0, l_1, l_2)$ in $A_X$. For this we will use that $([M],[F])\in \wMM_X$, and that by the main Theorem, \cite[Theorem 2.1]{BB22-2} this means that the scheme-theoretic intersection of $Y$, defined by the $2\times2$ minors of $S$, and $X$ contains a degree $2$ curve. This will give us enough of a connection between $F$ and the possible relations between $l_0, l_1, l_2$ modulo $F$ to conclude. 

First a preliminary observation: suppose that 
\[
\alpha l_2 +  \beta l_1 + \gamma  l_0 = g \cdot F
\]
for $\alpha, \beta, \gamma, g \in \C [x_0,\dots , x_4]$, and write 
\[
F = q_2 l_2 + q_1l_1 +q_0l_0 
\]
with $q_j$ some quadrics, which is always possible since $X$ in particular contains $l$. Then
\[
(\alpha -gq_2) l_2 + ( \beta-gq_1) l_1 + (\gamma -gq_0) l_0 = 0
\]
and in particular, all relations modulo $F$ between $l_2, l_1, l_0$ are generated by Koszul relations and the one relation $q_2 l_2 + q_1l_1 +q_0l_0 \equiv 0 \mod F$. Whenever $(\alpha, \beta , \gamma)$ gives a Koszul relation between $l_2, l_1, l_0$, clearly $\alpha m_2  + \beta m_1  + \gamma m_0$ is in the ideal generated by $(l_0, l_1, l_2)$. So it suffices to check if also
\[
q_2 m_2 + q_1m_1 +q_0m_0 \in (l_0, l_1, l_2)
\]
if we write $F = q_2 l_2 + q_1l_1 +q_0l_0$. We distinguish cases (d) and (e), starting with (e) which is a little simpler. 
In that case, the matrix $S^t$ takes the form 
\[
S^t = \begin{pmatrix}
	l_2 & -l_1 & l_0 & 0 & 0 & 0  \\
  0 & 0 & l_3 & l_2 & -l_1 & l_0 
\end{pmatrix}. 
\]
The unique pure-dimensional degree $2$ subscheme contained in $Y$ in this case is the plane double line with saturated ideal $(l_1, l_2, l_0^2)$. If $F=q_2 l_2 + q_1l_1 +q_0l_0$ is in that ideal, we must have $q_0l_0 \in (l_1, l_2, l_0^2)$, meaning $q_0$ is in $(l_1, l_2, l_0)$. Therefore, $q_2 m_2 + q_1m_1 +q_0m_0 = q_0l_3 \in (l_0, l_1, l_2)$ in this case. 

In case (d)
\[
S^t = \begin{pmatrix}
	l_2 & -l_1 & l_0 & 0 & 0 & 0  \\
  0 & -l_4 & l_3 & l_2 & -l_1 & l_0 
\end{pmatrix},
\]
and the saturation of the ideal of  $2\times 2$ minors in that case is the ideal of a double line on the quadric $l_2=0, l_0l_4-l_1l_3=0$. It is generated by $l_2, l_0l_4-l_1l_3, l_1^2, l_0l_1, l_0^2$. Again $F=q_2 l_2 + q_1l_1 +q_0l_0$. The linear forms $l_0, \dots , l_4$ are independent, and putting $l_2=0$, which amounts to working in the polynomial ring in $l_0, \dots , l_3$ we see that $\overline{F} = \overline{q}_1 l_1 +\overline{q}_0 l_0$ is in the ideal generated by 
\begin{align*}
(-l_3)l_1 + l_4 l_0 &  \\
 l_1^2 & \\
 l_0^2 & \\
 l_0l_1 
\end{align*}
where the overlines mean ``put $l_2=0$", or equivalently, work in the polynomial ring in $l_0, \dots , l_3$. Subtracting of a multiple of $(-l_3)l_1 + l_4 l_0$ we can assume that $\overline{F} = \overline{q}_1 l_1 +\overline{q}_0 l_0$ is even in the ideal generated by $l_1^2, l_0l_1, l_0^2$, meaning that the new $\overline{q}_1, \overline{q}_0$ are in $(l_1,l_0)$. Summarising, that means that we can write the original $q_1, q_0$ as
\begin{align*}
q_1 =&  \lambda_1 l_2 - \lambda l_3 + q_1' \\
q_0= & \lambda_2 l_2 + \lambda  l_4 + q_2' 
\end{align*}
with $\lambda_1, \lambda_2, \lambda$ some linear forms and $q_1', q_2' \in (l_1,l_0)$. 
Therefore also in this case, 
\[
q_2 m_2 + q_1m_1 +q_0m_0 = q_1 l_4 + q_0l_3 \in (l_0, l_1, l_2).
\]
\end{proof}

\begin{theorem}\label{tComparisonDruel}
There is an isomorphism
\[
\im\, \Pfad \colon \wMM_X \to \bMM_X .
\]
sending $([M], [F])$ to $\FF_M$. 
\end{theorem}

\begin{proof}
Lemmata \ref{lDruelNonZero}, \ref{lConic}, \ref{lTwoLines}, \ref{lNotPolystable} show that $\im\, \Pfad$ defines a morphism from $\wskss$ to $\bMM_X$ that is constant on $G$-orbits, giving a morphism $\wMM_X \to \bMM_X$. 

To show that $\im\, \Pfad$ is injective on closed points, we distinguish two cases. 
If $\EE \in \oMM_X\subset \bMM_X$, then $\EE$ fits into a short exact sequence 
\[
0 \to 6 \OO_{\P^4}(-2) \xrightarrow{M} 6 \OO_{\P^4}(-1) \to \EE \to 0 
\]
where $M$ is a skew $6\times 6$ matrix whose Pfaffian defines $X$. 
If 
\[
0 \to 6 \OO_{\P^4}(-2) \xrightarrow{M'} 6 \OO_{\P^4}(-1) \to \EE \to 0 
\]
is a second such sequence, then there exist invertible matrices $S,T \in \mathrm{GL_6}(\C)$ such that
\[
	M' = S^{-1}M T.
\]
By Lemma \ref{lAppendixGeneralSkew} this implies that there exists  an invertible matrix $U \in \mathrm{GL_6}(\C)$
such that
\[
M' = S^{-1} M T = U^t M U.
\]
Therefore $M$ and $M'$ are in the same $G$-orbit.

\medskip

If $\EE\in \bMM_X - \oMM_X = \mathcal{A}\cup \mathcal{B}$, then the assertion follows from the uniqueness statements in by Lemmata \ref{lTwoLines} and \ref{lConic}. 

Moreover, $\im\, \Pfad$ is dominant since all of $\oMM_X$ is in its image, hence bijective on closed points since $\wMM_X$ is projective. Since $\bMM_X$ is known to be smooth by Druel's results, the morphism $\im\, \Pfad $ must be an isomorphism by Zariski's main theorem.
\end{proof}


\


\appendix

\section{A linear algebra lemma}\label{sAppendixSkew}

\noindent
We want to prove the following

\begin{lemma} \label{lAppendixGeneralSkew}
Let $M$ be a skew-symmetric $n \times n$ matrix with entries in a $\C$-algebra $R$ and $A,B \in \mathrm{GL_n}(\C)$  invertible matrices such that $M'=A^{-1}MB$ is also skew. Then there exists an an invertible matrix
$S \in \mathrm{GL_n}(\C)$ such that
\[
	A^{-1}MB = S^t M S.
\]
\end{lemma}

\noindent
Since 
\[
	A^{-1}MB = S^t M S \iff M(BA^t) = (AS^t) M (SA^t)
\]
we can reduce to the case where $A = \id$.
Furthermore, we observe that for any $T \in \mathrm{GL_n}(\C)$
\[
	M' = MB \iff T^tM'T = (T^tMT)(T^{-1}BT).
\]
We can therefore assume that $B$ has Jordan normal form.

By the following proposition we can reduce to the case where all Jordan blocks of $B$ have
the same eigenvalue:

\begin{proposition} 
Let $M$ be skew and
assume that $B$ has Jordan blocks for pairwise different eigenvalues $\lambda_1, \dots , \lambda_n \in \C^*$. We write in block form 
\[
	MB = \begin{pmatrix} 
			M_{11} & \dots & M_{1n} \\
			\vdots & \ddots & \vdots \\
			-M_{1n}^t & \dots & M_{nn} 
		 \end{pmatrix} 
		 \begin{pmatrix}
		 B_{\lambda_1} & & \\
		  & \ddots & \\
		  & & B_{\lambda_n}
		 \end{pmatrix}
		 = M'
\]
with $B_{\lambda_i}$ a square matrix containing all the Jordan blocks for the eigenvalue $\lambda_i$ on the diagonal.
If $M'$ is also skew, then $M_{ij}=0$ for $i\neq j$. 
\end{proposition}

\begin{proof}
Choose indices $i< j$. 
If $M'$ is skew then
\[
	M'_{ij}= M_{ij}B_{\lambda_j} = -(M_{ji}')^t= - \bigl( -(M_{ij})^tB_{\lambda_i}\bigr)^t = B_{\lambda_i}^t M_{ij}.
\]
We want to prove $M_{ij} = 0$ by induction on the number of rows of $M_{ij}$. 

\medskip

\noindent
If $M_{ij}$ has only one row, then
\[
	M_{ij}B_{\lambda_j} = B_{\lambda_i}^t M_{ij} = \lambda_i M_{ij} = M_{ij} (\lambda_i \cdot\id)
\]
which implies
\[
	 M_{ij}(B_{\lambda_j}-\lambda_i\cdot\id) = 0.
\]
Since $\lambda_i \not=\lambda_j$ we get that $M_{ij}=0$.

\medskip
\noindent
If $M_{ij}$ has more than one row, we write the equation above as
\[
	 B_{\lambda_i}^t M_{ij}= \begin{pmatrix}
	(\widetilde{B}_{\lambda_i})^t & 0 \\
	\begin{pmatrix} 0 & \cdots  &  0 &  \epsilon \end{pmatrix}& \lambda_i
	\end{pmatrix}
	\begin{pmatrix}
	\widetilde{M}_{ij} \\
	m
	\end{pmatrix}
	= \begin{pmatrix}
	\widetilde{M}_{ij} \\
	m
	\end{pmatrix}
	B_{\lambda_j}
\]
with $\epsilon = 0$ or $1$. We see that $\widetilde{M}_{ij}$ satisfies the induction hypothesis and therefore $\widetilde{M}_{ij}=0$. The
equation above then reduces to
\[
	\lambda_i m = m B_{\lambda_j}
\]
which as in the one-row case implies $m=0$.
\end{proof}

\noindent
The case where $B$ has only one eigenvalue is treated by 

\begin{proposition}
Let $M$ be skew and $B_\lambda$ a matrix consisting of Jordan blocks with eigenvalue $\lambda \in \C^*$. If
\[
	M' = MB_\lambda
\]
is again skew, then there exists an invertible matrix $S$ such that $S^2=B_\lambda$ and
\[
	M' = S^t M S.
\]
\end{proposition}

\begin{proof}
Since $M$ and $M'$ are skew, we have
\[
	MB_{\lambda} = M' = -(M')^t = -B_{\lambda}^tM^t = B_{\lambda}^tM
\]
We can now write 
\[
	B_\lambda = \lambda \cdot \id + N
\]
with $N$ nilpotent. Plugging this into the above equation we get
\[
	M (\lambda \cdot \id + N) = (\lambda \cdot \id + N^t)M 
	\quad \iff \quad
	MN = N^tM.
\]	
For $s_i \in \C$ we consider the matrix 
\[
	S = \sum_i s_i N^i .
\]
Since the $s_i$ are in $\C$ we can successively solve the equation $S^2 = B_{\lambda}$ for the $s_i$, computing $s_0, s_1, s_2 \dots$ in this order. There are then two 
solutions corresponding to the two solutions of $s_0^2 = \lambda$ (after which $s_1, \dots$ are uniquely determined using $\lambda\neq 0$). Now since $MN = N^tM$ we also have
$MS = S^tM$. With this we get
\[
	M'=MB_{\lambda} = MS^2 = (MS) S = S^tMS
\]
as claimed.
\end{proof}

\providecommand{\bysame}{\leavevmode\hbox to3em{\hrulefill}\thinspace}
\providecommand{\href}[2]{#2}

\end{document}